\documentclass[a4paper,11pt]{article}
\usepackage{graphicx}
\usepackage{epstopdf}
\usepackage{amsthm}
\usepackage[utf8]{inputenc}
\usepackage[english]{babel}
\newtheorem{theorem}{Theorem}[section]

\theoremstyle{definition}
\newtheorem{definition}{Definition}[section]
\newtheorem*{remark}{Remark}
\usepackage{amssymb, amsmath}
\title{Subclass of bi-univalent function associated to Chebyshev polynomial}
\date{}
\begin{document}
\numberwithin{equation}{section}
\maketitle
\date{}
\begin{center}
	\author{{\bf{G. M. Birajdar}}\vspace{.31cm}\\
		Department of Mathematics \& Statistics,\\
		Dr. Vishwanath Karad MIT World Peace University, Pune (M.S) India 411038\\
		Email: gajanan.birajdar@mitwpu.edu.in}\vspace{.5cm}\\
	\author{{\bf{N. D. Sangle}}\vspace{.31cm}\\
		Department of Mathematics,\\
		D. Y. Patil College of Engineering \& Technology, Kolhapur, (M.S.) India 416006\\
		Email: navneet\_sangle@rediffmail.com}\vspace{.5cm}\\
\end{center}
\vspace{1cm}
\abstract{}
In this study, a subclass of bi-univalent functions is introduced by employing Chebyshev polynomials. Furthermore, coefficient bounds associated with Chebyshev polynomials, along with the Fekete–Szegö inequalities for functions belonging to these classes, are derived.\\



\section{Introduction} 

Bi-univalent functions constitute an important branch of geometric function theory and have attracted considerable attention due to their rich analytical structure and wide applicability. These functions, together with their inverse mappings, are useful in obtaining coefficient bounds, Fekete--Szeg\"o inequalities, Hankel determinants, and various inclusion properties for newly defined subclasses. They also play a significant role in conformal mapping, approximation theory, operator theory, and several engineering models. Bi-univalent functions are analytic functions defined in the open unit disk such that both function and its inverse function are univalent in their respective domains. These functions form an important subclass of geometric function theory and are widely studied for coefficient bounds, Fekete--Szeg\"o inequalities, Hankel determinants, and other structural properties. The study of bi-univalent functions extends classical univalent function theory by imposing injectivity on both the function and its inverse, making the coefficient problem more challenging and interesting.

Crossing symmetry is a fundamental concept in quantum field theory and scattering theory. It expresses the invariance of scattering amplitudes under the interchange of incoming and outgoing particles through suitable transformations of Mandelstam variables. This symmetry provides strong constraints on physical amplitudes and plays a central role in effective field theory, dispersion relations, and S-matrix analysis. Recent works have shown that methods from univalent function theory, such as coefficient estimates and growth bounds, can be connected with crossing-symmetric amplitudes and Wilson coefficient inequalities. Two-sided inequalities for ratios of Wilson coefficients were obtained in the framework of effective field theory in \cite{B18}.

Recent studies \cite{B18} have revealed interesting connections between univalent function theory and modern physics, particularly in effective field theory (EFT) and quantum field theory (QFT). The coefficient estimates, growth theorems, and univalence conditions have been linked with scattering amplitudes, Wilson coefficients, and crossing symmetric bounds arising in QFT. Such interactions show that geometric function theory can provide useful mathematical tools for analyzing physical models. 

Motivated by these developments, we introduce and investigate a new subclass of bi-univalent functions associated with the proposed operator. Several coefficient estimates and related properties are obtained, which further enrich the theory of bi-univalent functions and their applications.\\
Let $A$ denote the family of functions $g(z)$ that are analytic and univalent in the open unit disk 
\[
E=\{z\in \mathbb{C}:\lvert z\rvert<1\},
\]
and having the series representation
\begin{equation}
	g(z)=z+\sum_{n=2}^{\infty} d_n z^n,\qquad d_n\geq 0,
	\label{f}
\end{equation}
where the normalization conditions are given by $g(0)=0$ and $g'(0)=1$.\\
Let $S$ represent the subclass of $A$ containing functions expressed in the form  of equation \eqref{f}.\\
In \cite{B1}, Salagean introduced the following differential operator:
\begin{align*}
	D^{0}g(z) &= g(z),\\
	D^{1}g(z) &= zg'(z),
\end{align*}
and, in general,
\begin{align*}
	D^{n}g(z)=D\bigl(D^{n-1}g(z)\bigr), \qquad \left(n\in \mathbb{N}_{0}=\mathbb{N}\cup\{0\}\right).
\end{align*}
For a function $g$ given by the form \eqref{f}, it follows that
\begin{align*}
	D^{n}g(z)=z+\sum_{k=2}^{\infty} k^{n} d_{k} z^{k}.
\end{align*}
Each function $g \in S$ possesses an inverse function $g^{-1}$ satisfying
\[
g^{-1}(g(z))=z \qquad (z\in E)
\]
and
\[
g\bigl(g^{-1}(\varpi)\bigr)=\varpi, \qquad \left(|\varpi|<r_{0}(g),\; r_{0}(g)\geq \frac{1}{4}\right).
\]
Furthermore, the inverse function admits the expansion
\begin{equation}
	g^{-1}(\varpi)=\varpi-d_{2}\varpi^{2}+\left(2d_{2}^{2}-d_{3}\right)\varpi^{3}
	-\left(5d_{2}^{3}-5d_{2}d_{3}+d_{4}\right)\varpi^{4}+\cdots .
\end{equation}
A function $g \in S$ is called bi-univalent in $E$ whenever both $g(z)$ and its inverse $g^{-1}(z)$ are univalent in $E$.\\
The family of bi-univalent functions, denoted by $\Sigma$, was introduced by Lewin \cite{B3}, who established the estimate $\lvert d_{2}\rvert < 1.51$.
 Many researchers have introduced various subclasses of bi-univalent functions and investigated the Taylor--Maclaurin coefficients $\lvert d_{2}\rvert$ and $\lvert d_{3}\rvert$ for functions belonging to these subclasses (see \cite{B5,B6,B7,B8,B9,B10,B11}).\\
 
 Let $g$ and $h$ be analytic in $E$. Then $g$ is said to be subordinate to $h$ in $E$, denoted by
 \[
 g \prec h, \qquad z\in E,
 \]
 if there exists an analytic function $\varphi$ in $E$ satisfying $\varphi(0)=0$ and $\lvert \varphi(z)\rvert<1$ for every $z\in E$, such that
 \[
 g(z)=h\bigl(\varphi(z)\bigr), \qquad z\in E.
 \]
In numerical analysis, Chebyshev polynomials have gained significant attention from both theoretical and applied perspectives. Moreover, a considerable number of researchers have focused on the study of orthogonal Chebyshev polynomials. Further properties and applications of the Chebyshev polynomials $T_n(t)$ and $U_n(t)$ may be found in \cite{B13,B14}.\\
For $-1<t<1$, the Chebyshev polynomials of the first and second kinds, denoted respectively by $T_n(t)$ and $U_n(t)$, are given by
\begin{align*}
	T_n(t)&=\cos(n\alpha),
\end{align*}
and
\begin{align*}
	U_n(t)&=\frac{\sin\bigl((n+1)\alpha\bigr)}{\sin \alpha},
\end{align*}
where $n$ represents the degree of the polynomial and $t=\cos \alpha$.\\

If we take $t=\cos \beta$, with
\[
-\frac{\pi}{3}<\beta<\frac{\pi}{3},
\]
then the generating function can be written as
\begin{align*}
	H(z,t)=\frac{1}{1-2tz-z^{2}}
	=1+\sum_{m=1}^{\infty}\frac{\sin\bigl((m+1)\beta\bigr)}{\sin \beta}\,z^{m},
	\qquad (z\in E).
\end{align*}
 Therefore,
 \begin{align*}
 	H(z,t)=1+2\cos\beta\, z+\left(3\cos^{2}\beta-\sin^{2}\beta\right)z^{2}+\cdots,
 	\qquad (z\in E).
 \end{align*}
 
 Equivalently, it may be expressed as
 \begin{align*}
 	H(z,t)=1+U_{1}(t)z+U_{2}(t)z^{2}+\cdots,
 	\qquad \left(z\in E,\,-1<t<1\right),
 \end{align*}
 where
 \begin{align*}
 	U_{m-1}(t)=\frac{\sin\!\left(m\cos^{-1}t\right)}{\sqrt{1-t^{2}}},
 	\qquad m\in \mathbb{N}_{0}=\{0,1,2,\ldots\},
 \end{align*}
 are the Chebyshev polynomials of the second kind.
The generating function corresponding to the Chebyshev polynomials $T_m(t)$ of the first kind is given by
\begin{align*}
	\sum_{m=0}^{\infty} T_m(t)\,z^m
	=
	\frac{1-tz}{1-2tz-z^2},
	\qquad (z\in E).
\end{align*}

The Chebyshev polynomials $T_m(t)$ and $U_m(t)$ satisfy the relations
\begin{align*}
	\frac{d}{dt}T_m(t)&=mU_{m-1}(t),\\
	T_m(t)&=U_m(t)-tU_{m-1}(t),\\
	2T_m(t)&=U_m(t)-U_{m-2}(t).
\end{align*}

Moreover,
\begin{align*}
	U_1(t)&=2t,\\
	U_2(t)&=4t^2-1,\\
	U_3(t)&=8t^3-4t,
\end{align*}
and the recurrence relation is
\begin{align*}
	U_m(t)=2t\,U_{m-1}(t)-U_{m-2}(t).
\end{align*}
 
Inspired by the works of Dziok et al.~\cite{B15} and S.~Altinkaya and S.~Yalcin~\cite{B16,B17}, we employ the expansions of Chebyshev polynomials to derive bounds for the initial coefficients of certain subclasses of bi-univalent functions. In addition, the Fekete--Szeg\"o inequality is obtained for the class $G_{\sum}(\delta,t,m)$.

\section{Coefficient Bounds for the Class  $G_{\sum} (\delta, t, m)$ } 
We now define the class $G_{\Sigma}(\delta,t,m)$ of bi-univalent functions as follows.

\begin{definition}
	Let $g\in \Sigma$ be of the form given in equation \eqref{f}. Then $g$ is said to belong to the class $G_{\Sigma}(\delta,t,m)$, where $\delta\geq 1$, $t\in \left(\frac{1}{2},1\right)$, and $z,w\in E$, if the following subordination conditions hold:
	\begin{equation}
		\frac{(1-\delta)D^{m}g(z)+\delta D^{m+1}g(z)}{z}
		\prec H(z,t)
		\qquad z\in E,
	\end{equation}
	and
	\begin{equation}
		\frac{(1-\delta)D^{m}h(\varpi)+\delta D^{m+1}h(\varpi)}{\varpi}
		\prec H(\varpi,t)
		\qquad \varpi\in E,
	\end{equation}
	where $h=g^{-1}$.
\end{definition}
\begin{theorem}
Suppose that the function $g$, given by equation \eqref{f}, belongs to the class $G_{\Sigma}(\delta,t,m)$. Then the following coefficient estimates hold:
\begin{align}
	|d_2|
	&\le
	\frac{2t\sqrt{2t}}
	{\sqrt{
			4t^2\left[\left(1-\delta\right)3^m+\delta 3^{m+1}\right]^2
			-
			\left(8t^2-2\right)\left[\left(1-\delta\right)2^m+\delta 2^{m+1}\right]
	}},
\end{align}
and
\begin{align}
	|d_3|
	&\le
	\frac{2t}{\left(1-\delta\right)3^m+\delta 3^{m+1}}
	+
	\frac{16t^2}
	{\left[\left(1-\delta\right)2^m+\delta 2^{m+1}\right]^2}.
\end{align}
\end{theorem}
\begin{proof}
Let $g \in G_{\Sigma}(\delta,t,m)$ be represented by the Taylor--Maclaurin expansion in equation \eqref{f}. Then there exist analytic functions $u$ and $v$ in $E$ satisfying
\[
\mu(0)=\nu(0)=0,\qquad |\mu(z)|<1,\qquad |
\nu(\varpi)|<1,\qquad (z, \varpi\in E),
\]
so that we may express
\begin{align}
	\frac{(1-\delta)D^{m}g(z)+\delta D^{m+1}g(z)}{z}
	&\prec \frac{1}{1-2tz-z^{2}},
	\qquad z\in E,
\end{align}
and
\begin{align}
	\frac{(1-\delta)D^{m}h(\varpi)+\delta D^{m+1}h(\varpi)}{\varpi}
	&\prec \frac{1}{1-2t\varpi-\varpi^{2}},
	\qquad \varpi\in E.
\end{align}

Equivalently, we may write
\begin{equation}
	\frac{(1-\delta)D^{m}g(z)+\delta D^{m+1}g(z)}{z}
	=
	1+U_{1}(t)\mu(z)+U_{2}(t)u^{2}(z)+\cdots,
\end{equation}
and
\begin{equation}
	\frac{(1-\delta)D^{m}h(\varpi)+\delta D^{m+1}h(\varpi)}{\varpi}
	=
	1+U_{1}(t)
\mu(\varpi)+U_{2}(t)\nu^{2}(\varpi)+\cdots.
\end{equation}

Using the series forms of $\mu$ and $\nu$, we obtain
\begin{equation}
	\frac{(1-\delta)D^{m}g(z)+\delta D^{m+1}g(z)}{z}
	=
	1+U_{1}(t)\xi_{1}z+\bigl[U_{1}(t)\xi_{2}+U_{2}(t)\xi_{1}^{2}\bigr]z^{2}+\cdots,
\end{equation}
and
\begin{equation}
	\frac{(1-\delta)D^{m}h(\varpi)+\delta D^{m+1}h(\varpi)}{\varpi}
	=
	1+U_{1}(t)\eta_{1}\varpi+\bigl[U_{1}(t)\eta_{2}+U_{2}(t)\eta_{1}^{2}\bigr]\varpi^{2}+\cdots.
\end{equation}

Here,
\[
|\mu(z)|=\left|\xi_{1}z+\xi_{2}z^{2}+\cdots\right|<1,
\qquad
|
\mu(\varpi)|=\left|\eta_{1}\varpi+\eta_{2}\varpi^{2}+\cdots\right|<1,
\]
which implies that
\[
|\xi_{j}|<1,\qquad |\eta_{j}|<1,\qquad (j\in\mathbb{N}).
\]
It follows directly from equations (2.9) and (2.10) that
\begin{align}
	U_{1}(t)\xi_{1}
	&=
	\left[\left(1-\delta\right)2^{m}+\delta 2^{m+1}\right]d_{2},
	\\
	U_{1}(t)\xi_{2}+U_{2}(t)\xi_{1}^{2}
	&=
	\left[\left(1-\delta\right)3^{m}+\delta 3^{m+1}\right]d_{3},
	\\
	-U_{1}(t)\eta_{1}
	&=
	\left[\left(1-\delta\right)2^{m}+\delta 2^{m+1}\right]d_{2},
\end{align}
and
\begin{align}
	U_{1}(t)\eta_{2}+U_{2}(t)\eta_{1}^{2}
	=
	\left[\left(1-\delta\right)3^{m}+\delta 3^{m+1}\right]
	\left(2d_{2}^{2}-d_{3}\right).
\end{align}
From equations (2.11) and (2.13), we obtain
\begin{align}
	\xi_{1}=-\eta_{1},
\end{align}
and hence
\begin{align}
	2\left[\left(1-\delta\right)2^{m}+\delta 2^{m+1}\right]^{2}\varpi_{2}^{2}
	&=
	\left[U_{1}(t)\right]^{2}\left(\xi_{1}^{2}+\eta_{1}^{2}\right),
	\\
	d_{2}^{2}
	&=
	\frac{\left[U_{1}(t)\right]^{2}\left(\xi_{1}^{2}+\eta_{1}^{2}\right)}
	{2\left[\left(1-\delta\right)2^{m}+\delta 2^{m+1}\right]^{2}}.
\end{align}

Adding equations (2.12)and (2.14), we get
\begin{align}
	2\left[\left(1-\delta\right)3^{m}+\delta 3^{m+1}\right]d_{2}^{2}
	=
	U_{1}(t)\left(\xi_{2}+\eta_{2}\right)+U_{2}(t)\left(\xi_{1}^{2}+\eta_{1}^{2}\right).
\end{align}

Substituting the value of $\xi_{1}^{2}+\eta_{1}^{2}$ from the previous relation, we obtain
\begin{align}
	d_{2}^{2}
	=
	\frac{\left[U_{1}(t)\right]^{3}\left(\xi_{2}+\eta_{2}\right)}
	{2\left[U_{1}(t)\right]^{2}\left[\left(1-\delta\right)3^{m}+\delta 3^{m+1}\right]^{2}
		-
		2U_{2}(t)\left[\left(1-\delta\right)2^{m}+\delta 2^{m+1}\right]^{2}}.
\end{align}

Therefore,
\begin{align}
	|d_{2}|
	\le
	\frac{2t\sqrt{2t}}
	{\sqrt{
			4t^{2}\left[\left(1-\delta\right)3^{m}+\delta 3^{m+1}\right]^{2}
			-
			\left(8t^{2}-2\right)\left[\left(1-\delta\right)2^{m}+\delta 2^{m+1}\right]^{2}
	}}.
\end{align}
By subtracting equation (2.14) from (2.12), we obtain
\begin{align*}
\left[ {\left( {1 - \delta } \right){3^m} + \delta {3^{m + 1}}} \right]\left( {2{d_3} - 2d_2^2} \right) = {U_1}(t)\left( {{\xi_2} - {\eta_2}} \right) + {U_2}(t)\left( {\xi_1^2 - \eta_1^2} \right).
\end{align*}
Using equation (2.15), we get
\begin{align}
{d_3} = \frac{{{U_1}(t)\left( {{\xi_2} - {\eta_2}} \right)}}{{2\left[ {\left( {1 - \delta } \right){3^m} + \delta {3^{m + 1}}} \right]}} + d_2^2.
\end{align}
By using equation (2.17) in (2.21), we obtain
\begin{align*}
{d_3} = \frac{{{U_1}(t)\left( {{\xi_2} - {\eta_2}} \right)}}{{2\left[ {\left( {1 - \delta } \right){3^m} + \delta {3^{m + 1}}} \right]}} + \frac{{{{\left[ {{U_1}(t)} \right]}^2}\left( {\xi_1^2 + \eta_1^2} \right)}}{{2{{\left[ {\left( {1 - \delta } \right){2^m} + \delta {2^{m + 1}}} \right]}^2}}}
\end{align*}
which gives
\begin{align*}
\left| {{d_3}} \right| \le \frac{{2t}}{{\left( {1 - \delta } \right){3^m} + \delta {3^{m + 1}}}} + \frac{{16{{t}^2}}}{{{{\left[ {\left( {1 - \delta } \right){2^m} + \delta{2^{m + 1}}} \right]}^2}}}.
\end{align*}
\end{proof}
\begin{theorem}
	Let $g$, given by the form in equation \eqref{f}, belong to the class $G_{\Sigma}(\delta,t,m)$. Then
	\[
	\left| {{d_3} - rd_2^2} \right| = 
	\begin{cases}
	\frac{{\left| {{U_1}(t)} \right|}}{{\left[ {\left( {1 - \delta } \right){3^m} + \delta {3^{m + 1}}} \right]}} & \mbox{if} \quad 0\le \left| {\sigma (r,t)} \right| \le \frac{1}{{2\left[ {\left( {1 - \delta} \right){3^m} + \delta {3^{m + 1}}} \right]}} \\
	\\
	\frac{{2\left| {{U_1}(t)} \right|\left| {\sigma (r,t)} \right|}}{{\left[ {\left( {1 - \delta } \right){3^m} + \delta {3^{m + 1}}} \right]}} & \mbox{if} \quad \left| {\sigma (r,t)} \right| \ge \frac{1}{{2\left[ {\left( {1 - \delta } \right){3^m} + \delta {3^{m + 1}}} \right]}}
	\end{cases}
	\]
	where 
	\begin{align*}
	\sigma (r,t) = \frac{{(1 - r){{\left[ {{U_1}(t)} \right]}^2}}}{{2\left[ {{{\left[ {{U_1}(t)} \right]}^3}{{\left[ {\left( {1 - \delta } \right){3^m} + \delta{3^{m + 1}}} \right]}^2} - {U_2}(t)\left[ {\left( {1 - \delta } \right){2^m} + \delta {2^{m + 1}}} \right]} \right]}}.
	\end{align*}
\end{theorem}
\begin{proof}
From equation (2.21), for  $t \in \mathbb{R}$, we write
\begin{align}
{d_3} - rd_2^2 = \frac{{{U_1}(t)\left( {{\xi_2} - {\eta_2}} \right)}}{{2\left[ {\left( {1 - \delta } \right){3^m} + \delta {3^{m + 1}}} \right]}} + (1 - r)d_2^2.
\end{align}

By substituting equation (2.19) in equation (2.22), we have
\begin{align*}
{d_3} - rd_2^2 &= \frac{{{U_1}(t)\left( {{\xi_2} - {\eta_2}} \right)}}{{2\left[ {\left( {1 - \delta } \right){3^m} + \delta {3^{m + 1}}} \right]}} + (1 - r)\frac{{{{\left[ {{h_2}(x)} \right]}^3}\left( {{\xi_2} + {\eta_2}} \right)}}{{2\left[ {{{\left[ {{U_1}(t)} \right]}^3}{{\left[ {\left( {1 - \delta } \right){3^m} + \delta {3^{m + 1}}} \right]}^2} - {U_2}(t)\left[ {\left( {1 - \delta } \right){2^m} + \delta{2^{m + 1}}} \right]} \right]}}\\
&= {U_1}(t)\left\{ {\left( {\sigma (r,t) + \frac{1}{{2\left[ {\left( {1 - \delta } \right){3^m} + \delta {3^{m + 1}}} \right]}}} \right){\xi_2} + \left( {\sigma (r,t) - \frac{1}{{2\left[ {\left( {1 - \delta } \right){3^m} + \delta {3^{m + 1}}} \right]}}} \right){\eta_2}} \right\}
\end{align*}
where
\begin{align*}
\sigma (r,t) = \frac{{(1 - r){{\left[ {{U_1}(t)} \right]}^2}}}{{2\left[ {{{\left[ {{U_1}(t)} \right]}^3}{{\left[ {\left( {1 - \delta } \right){3^m} + \delta{3^{m + 1}}} \right]}^2} - {U_2}(t)\left[ {\left( {1 - \delta } \right){2^m} + \delta {2^{m + 1}}} \right]} \right]}}.
\end{align*}
Hence, we conclude that
\[
\left| {{d_3} - rd_2^2} \right| = 
\begin{cases}
\frac{{\left| {{U_1}(t)} \right|}}{{\left[ {\left( {1 - \delta } \right){3^m} + \delta {3^{m + 1}}} \right]}} & \mbox{if} \quad 0\le \left| {\sigma (r,t)} \right| \le \frac{1}{{2\left[ {\left( {1 - \delta} \right){3^m} + \delta {3^{m + 1}}} \right]}} \\
\\
\frac{{2\left| {{U_1}(t)} \right|\left| {\sigma (r,t)} \right|}}{{\left[ {\left( {1 - \delta } \right){3^m} + \delta {3^{m + 1}}} \right]}} & \mbox{if} \quad \left| {\sigma (r,t)} \right| \ge \frac{1}{{2\left[ {\left( {1 - \delta } \right){3^m} + \delta {3^{m + 1}}} \right]}}.
\end{cases}
\]
\end{proof}
\begin{remark}
	When $\delta=1$, Theorem 2.2 takes the following form:
	
	Let $g$, given by equation (1.1) belong to the class $G_{\Sigma}(t,m)$. Then
	\[
	\left|d_{3}-r d_{2}^{2}\right|=
	\begin{cases}
		\dfrac{|U_{1}(t)|}{3^{m+1}}, & \text{if } 0\leq |\sigma(r,t)|\leq \dfrac{1}{2\cdot 3^{m+1}},\\[10pt]
		\dfrac{2|U_{1}(t)||\sigma(r,t)|}{3^{m+1}}, & \text{if } |\sigma(r,t)|\geq \dfrac{1}{2\cdot 3^{m+1}},
	\end{cases}
	\]
	where
	\[
	\sigma(r,t)=
	\frac{(1-r)[U_{1}(t)]^{2}}
	{2\left([U_{1}(t)]^{3}3^{2m+2}-U_{2}(t)2^{m+1}\right)}.
	\]
	
	This gives the Fekete--Szeg\"o inequality for the subclass $G_{\Sigma}(t,m)$.
\end{remark}

\section*{Conclusion}

A new subclass $G_{\Sigma}(\delta,t,m)$ of bi-univalent functions was introduced by using the Salagean operator and Chebyshev polynomials. Estimates for the initial coefficients $|d_2|$ and $|d_3|$ were derived, and the Fekete--Szeg\"o inequality was established for this class. These results generalize existing findings and highlight the importance of Chebyshev polynomials in geometric function theory.
\section*{Conflict of Interest}

The authors declare that there is no conflict of interest regarding the publication of this paper.
\pagebreak


\begin{thebibliography}{99}
	
	\bibitem{B1}
	G. S. Salagean, Subclasses of univalent functions, 
	\textit{Lecture Notes Math.} \textbf{1013}(1983), 362--372.
	
	\bibitem{B2}
	D. A. Brannan and T. S. Taha, On some classes of bi-univalent functions, 
	\textit{Studia Univ. Babe\c{s}-Bolyai Math.} \textbf{31}(1986), 70--77.
	
	\bibitem{B3}
	M. Lewin, On a coefficient problem for bi-univalent functions, 
	\textit{Proc. Amer. Math. Soc.} \textbf{18}(1967), 63--68.
	
	\bibitem{B4}
	P. L. Duren, \textit{Univalent Functions}, 
	Grundlehren der Mathematischen Wissenschaften, Vol. 259, Springer, New York, 1983.
	
	\bibitem{B5}
	H. M. Srivastava and D. Bansal, Coefficient estimates for a subclass of analytic and bi-univalent functions, 
	\textit{J. Egypt. Math. Soc.}(2014), 1--4.
	
	\bibitem{B6}
	S. Porwal and M. Darus, On a new subclass of bi-univalent functions, 
	\textit{J. Egypt. Math. Soc.} \textbf{21}(2013), 190--193.
	
	\bibitem{B7}
	H. M. Srivastava, S. B. Joshi, S. S. Joshi and H. Pawar, Coefficient estimates for certain subclasses of meromorphically bi-univalent functions, 
	\textit{Palestine J. Math.} \textbf{5}(2016), 250--258.
	
	\bibitem{B8}
	J. M. Jahangiri and S. G. Hamidi, Coefficient estimates for certain classes of bi-univalent functions, 
	\textit{Int. J. Math. Math. Sci.}(2013), Article ID 190560, 1--4.
	
	\bibitem{B9}
	B. A. Frasin and M. K. Aouf, New subclasses of bi-univalent functions, 
	\textit{Appl. Math. Lett.} \textbf{24}(2011), 1569--1573.
	
	\bibitem{B10}
	N. Magesh and J. Yamini, Coefficient bounds for a certain subclass of bi-univalent functions, 
	\textit{Int. Math. Forum} \textbf{8}(2013), 1337--1344.
	
	\bibitem{B11}
	S. Altinkaya and S. Yalcin, Coefficient estimates for two new subclasses of bi-univalent functions with respect to symmetric points, 
	\textit{J. Funct. Spaces}(2015), Article ID 145242, 1--5.
	
	\bibitem{B13}
	E. H. Doha, The first and second kind Chebyshev coefficients of the moments of the general-order derivative of an infinitely differentiable function, 
	\textit{Int. J. Comput. Math.} \textbf{51}(1994), 21--35.
	
	\bibitem{B14}
	J. C. Mason, Chebyshev polynomial approximations for the L-membrane eigenvalue problem, 
	\textit{SIAM J. Appl. Math.} \textbf{15}(1967), 172--186.
	
	\bibitem{B15}
	J. Dziok, R. K. Raina and J. Sokol, Application of Chebyshev polynomials to classes of analytic functions, 
	\textit{C. R. Acad. Sci. Paris Ser. I} \textbf{353}(2015), 433--438.
	
	\bibitem{B16}
	S. Altinkaya and S. Yalcin, On the Chebyshev polynomial coefficient problem of some subclasses of bi-univalent functions, 
	\textit{Gulf J. Math.}(2017), 1464--1473.
	
	\bibitem{B17}
	S. Altinkaya and S. Yalcin, On the Chebyshev polynomial bounds for classes of univalent functions, 
	\textit{Khayyam J. Math.} \textbf{2}(2016), 1--5.
	
	\bibitem{B18}
	P. Haldar, A. Sinha and A. Zahed, Quantum field theory and the Bieberbach conjecture, 
	\textit{SciPost Phys.} \textbf{11}(2021), 002.
	
\end{thebibliography}
\end{document}